\documentclass[reqno]{amsart}

\usepackage{amssymb}
\usepackage[stretch=10,shrink=10]{microtype}
\usepackage[showframe=false, margin = 1.5 in]{geometry}
\usepackage{mathtools}
\usepackage[shortlabels]{enumitem}
\usepackage{theoremref}
\usepackage{tikz}
\usepackage{subfigure, float}
\usepackage{comment}

\renewcommand{\emptyset}{\varnothing}
\newcommand{\E}{\mathbb{E}}

\newcommand{\rootvertex}{\varnothing}

\DeclareMathOperator{\RFM}{RFM}
\DeclareMathOperator{\SFM}{SFM}
\DeclareMathOperator{\FM}{FM}

\newtheorem{thm}{Theorem}[section]
\newtheorem{lemma}[thm]{Lemma}
\newtheorem{prop}[thm]{Proposition}

\newtheorem{conjecture}[thm]{Conjecture}

\theoremstyle{remark}
\newtheorem{remark}[thm]{Remark}
\theoremstyle{definition}

\newtheorem*{example}{Example}

\title{On the minimal drift for recurrence in the frog model on $d$-ary trees}
\author{Chengkun Guo}
\email{chg217@lehigh.edu}
\author{Si Tang}
\email{sit218@lehigh.edu}

\author{Ningxi Wei}
\email{niw318@lehigh.edu}

\address{Department of Mathematics, Lehigh University}

\begin{document}

\allowdisplaybreaks

\begin{abstract} We study the recurrence of one-per-site frog model
  $\FM(d, p)$ on a $d$-ary tree with drift parameter $p\in [0,1]$,
  which determines the bias of frogs' random walks.  We are interested
  in the minimal drift $p_{d}$ so that the frog model is
  recurrent. Using a coupling argument together with a generating
  function technique, we prove that for all $d \ge 2$, $p_{d}\le 1/3$,
  which is the optimal universal upper bound.
\end{abstract}

\maketitle
\section{Introduction}
Let $\mathbb T_{d}$ be a $d$-ary rooted tree where each vertex has $d$
child vertices and one parent vertex except for the root $\rootvertex$
which does not have a parent vertex. We study the standard
one-per-site frog model on $\mathbb T_{d}$ with drift parameter $p$,
which we denote by FM$(d, p)$. The model is defined as follows. At
time $t=0$, each vertex of $\mathbb T_{d}$ other than the root is
occupied by a sleeping frog (``inactive''); the root has a conscious
frog (``active'') at time $t=0$. Active frogs perform independent
random walks according to the following rules: (a) from vertices other
than the root, frogs make ``upward'' jumps (to the parent of the
current vertex) with probability $p$ and ``downward'' jumps (to child
vertices) with probability $(1-p)/d$; and (b) from the root vertex,
frogs always jump downward to one of the child vertices, with
probability $1/d$ each. Whenever an active frog visits a site with a
sleeping frog,  the sleeping frog wakes and begins its own independent
random walk,  following the same rules.  Let $V_{d, p}$ denote the total number of visits to the root in
FM$(d, p)$. It is known that $P(V_{d,p}=\infty)$ is either 0
or 1 \cite{kosygina01, HJJ1}, marking the two phases of FM$(d, p)$, {\em
  transience} and {\em recurrence}, respectively.

The recurrence of the frog model has been studied in various settings.
It depends on the graph structure (e.g., the degree distribution of
vertices), the distribution of the number of frogs initially at each
vertex, and the law of the random walk. Gantert and Schmidt
\cite{nina_drift1} showed that, on the integer lattice $\mathbb Z$, if
the numbers of sleeping frogs on the different vertices are i.i.d. copies of a
random variable $\eta$ and the random walk has a nonzero drift, then
the process is recurrent if and only if $\E(\log \eta)_{+} = \infty$
no matter how strong the drift is. The asymptotic behavior of the range
of the random walks was studied in \cite{ghosh_drift}  in the
transient case. In higher dimensions $d\ge 2$, a similar condition
$\E(\log \eta)^{(d+1)/2}_{+} = \infty$ was proved to be sufficient for
recurrence \cite{dobler_drift}. However, unlike the $d=1$ case, the
frog model with one-per-site setting can be either  transient or
recurrent, where the preference in choosing one direction (e.g.,
$\pm e_1$) and the net drift in the preferred direction are the key
determining factors \cite{nina_drift2}.

In \cite{HJJ1}, Hoffman, Johnson and Junge studied the recurrence of the
one-per-site frog model $\FM_d$ on the $d$-ary tree $\mathbb T_d$, where
active frogs  perform simple, nearest-neighbor random walks. This is the special case
of the frog model $\FM(d, p)$ where $p=1/(d+1)$. They
showed that $\FM_d$ is recurrent for $d=2$ and transient when
$d\ge 5$; it is currently unknown if $\FM_3$ and $\FM_4$ are
recurrent. It was later improved in  \cite{HJJ2, JJ3_log, JJ3_order} that if the
``one-per-site'' setting in $\FM_{d}$ is changed to ``$\Omega(d)$
frogs on each vertex'', then the frog model becomes recurrent for all
$d$.

In this work, we focus on the one-per-site setting. Since $\FM_{d}$ is
transient when $d\ge 5$, it is natural to ask what is the minimal
drift needed for a frog model on a $d$-ary tree to be
recurrent. Define
\[
p_{d}:=\inf\{p: \FM(d, p) \text{ is recurrent}\}.
\]
It is easy to see that $p_{d}\le 1/2$, as simple random walk on
$\mathbb Z$ is recurrent. By estimating the expected number of frogs
that can reach the root, we see that $p_{d}$ must be at least
$1/(d+1)$ for all $d$. Since $\FM_{2} = \FM(2,1/3)$ is recurrent
\cite{HJJ1}, implying $p_{2}\le 1/3$, we thus can conclude
$p_{2}=1/3$.

A better universal lower bound for $p_{d}$ can be obtained by
dominating $\FM(d, p)$ with a branching random walk on
$\mathbb Z_{\ge 0}$ where particles perform independent $p$-biased
random walk and split to two whenever moving to the right. This
branching random walk is equivalent to $\FM(\infty, p)$, in which a
new active frog is added to the process every time a frog moves away
from the root. When $p<q^{\ast}:=\frac{2-\sqrt{2}}{4}$, the branching
random walk is transient \cite{ECP19-frog, HJJ1} and therefore so is
the frog model $\FM(d, p)$. Thus, $p_{d} \ge q^{\ast}$ for all
$d\ge 2$. By refining the branching random walk to approximate two
steps of frog model, one can also show that $p_{d}-q^{\ast} > 0$ for
all $d\ge 2$. Since the branching random walk can be considered as a
frog model with $d=\infty$, it is natural to ask if
$p_{d}\to q^{\ast}$ as $d\to \infty$. A more basic question is
whether $(p_{d})_{d\ge 2}$ converges at all. It is natural to expect
that $p_{d+1} \le p_{d}$, since it would seem that there would be more
active frogs in $\FM(d+1, p)$ than in $\FM(d, p)$. But although
monotonicity of $p_{d}$ is believed to be true for the frog model on
homogeneous trees (as conjectured in \cite{improved}), it is not
generally true for other types of nested graphs \cite{mono}.

\begin{conjecture} 
\label{thm:monotone}
For all $d\ge 2$, $p_{d+1} \le p_{d}$.
\end{conjecture}

Direct coupling of two frog models $\FM(d, p)$ and $\FM(d', p)$
appears to be difficult, except for some special cases, for example
when $d'$ is an integer multiple of $d$ \cite[Proposition
1.2]{ECP19-frog}. In this case, it was proved that $p_{kd}\le p_{d}$
for all $d\ge 2, k\ge 1$, implying the convergence of
$(p_{d})_{d\ge 1}$ along certain subsequences. Unfortunately, the coupling 
can not generalize to other pairs of degrees.

Since $p_{2}=1/3$, Conjecture \ref{thm:monotone}, if true, would imply
that the universal upper bound of $p_{d}$ is $1/3$. Several
improvements to the trivial bound $p_{d}\le 1/2$ have been made: in
\cite{mono}, it was shown $p_{d}\le (d+1)/(2d-2)$, and  in
\cite{ECP19-frog} that $p_{d}\le 0.4155$. But none of these results is
close to the conjectured bound $1/3$. Our main contribution here is
to prove the sharp upper bound for $p_{d}$, i.e.,
\begin{thm} 
\label{thm:upbd}
For all $d\ge 2$, $p_{d} \le \frac{1}{3}$.
\end{thm}
\begin{remark} Since $p_{2}=1/3$, this result implies
  $p_{3}\le p_{2}$, as predicted by Conjecture
  \ref{thm:monotone}. More importantly, this inequality is the first such 
  result that compares two frog models where neither 
  tree degree is an integer multiple of the other, bypassing the 
  difficulty of directly coupling two frog models. 
\end{remark}

\subsection{Proof Strategy}
It suffices to prove that $\FM(d, 1/3)$ is recurrent, 
and we will construct a frog process $\mathcal P$ on
$\mathbb T_{d}$ that is dominated by $\FM(d, 1/3)$, so that if
$\mathcal P$ is recurrent (i.e., there are infinitely many visits to the root),
then $\FM(d, 1/3)$ must also be recurrent. This strategy was used in
\cite{HJJ1} to prove $\FM(2, 1/3)$ is recurrent; there the
dominated process $\mathcal P$ was the self-similar frog model
$\SFM(2, 1/3)$. This process is a modification of the ordinary frog model in
which some active frogs are removed, resulting in a self-similar
structure;  a precise definition is given in Section
\ref{sec:proof-prop-coupling}.  In \cite{ECP19-frog}, when proving a universal upper bound for $p_{d}$, the authors made use of a recursive frog model $\text{RFM}(d, p)$, where one can directly compare $\text{RFM}(d, p)$ with $\text{RFM}(d+1, p)$ and the critical drift $p'_{d}$ needed for $\RFM(d, p)$ to be recurrent appears to be monotone in $d$. Unfortunately, since too many frogs are removed in $\RFM$, it is very difficult (i.e., a strong drift is needed) for $\RFM$ to be recurrent. To this end, the authors of \cite{ECP19-frog} were unable to obtain a sharp upper bound for $p_{d}$.

Here we will also use a comparison to the self-similar frog model
$\SFM(d, p^{\ast})$ but with a more careful choice of $p^{\ast}$. The
correct choice is indicated by the following proposition.
We
need $\SFM(d, p^{\ast})$ to be dominated by $\FM(d, p)$ so that the
recurrence of $\SFM(d, p^{\ast})$ will imply that of $\FM(d,
p)$.

\begin{prop}
\label{prop:keycompare}
For $p \le 1/2$ and $d \ge 2$, if SFM$(d, p^{\ast})$ is recurrent,
then FM$(d, p)$ is also recurrent, where 
\begin{align}
\label{eqn:prop2}
p^{\ast} = p^\ast(d,p) = \frac{p(d-1)}{d-(d+1)p}.
\end{align}
\end{prop}

This will be proved by coupling arguments in Section
\ref{sec:proof-prop-coupling}. Although Proposition
\ref{prop:keycompare} holds generally for any $p\le 1/2$ and $d\ge 2$,
we will only need the result for $p=1/3$.

Once Proposition~\ref{prop:keycompare} has been proved,
we will then need to prove the
recurrence of $\SFM(d, p^{\ast})$ for 
$p^{\ast}=p^{\ast}(d, p)$ with $p=1/3$, that is, 
\[
p^{\ast}(d, 1/3) = \frac{\frac{1}{3}(d-1)}{d-(d+1)/3} = \frac{d-1}{2d-1}.
\]
This will be accomplished by the following proposition.
\begin{prop}
\label{prop:key}
The self-similar frog model SFM$(d, (d-1)/(2d-1))$ is recurrent.
\end{prop}  

The idea is to compare $\SFM(d, (d-1)/(2d-1))$ with $\SFM(2, 1/3)$,
which is  known to be recurrent \cite{HJJ1}.  To do this, let
$V^{\ast}_{d, p}$ be the total number of visits to the root in
$\SFM(d, p)$, and  consider the probability generating function
$g_{d, p}(x) := \E(x^{V^{\ast}_{d, p}})$ for $x\in[0, 1)$. We will use
the self-similar structure in $\SFM(d, p)$ to derive the following
self-consistency equation for $g_{d, p}(x)$
\begin{align}
\label{eqn:RDE}
g_{d, p}(x) = \mathcal A_{d, p}g_{d, p}(x),
\end{align}
where $\mathcal A_{d,p}$ is an operator on the set
$\mathcal I=\{f: [0,1) \to [0, 1], \text{nondecreasing}\}$ of
functions that will be defined in
Section~\ref{sec: proofofkey}. 
This is the most technical part of the paper: a recursive algorithm to
prove \eqref{eqn:RDE} for any $d$ is proposed in Section \ref{sec: proofofkey}.

When $d=2$ and $p=1/3$, it was shown in \cite{HJJ1} that 
\[
g_{2, 1/3}(x) = \mathcal A^{n}_{2, 1/3} g_{2, 1/3}(x) \le  \mathcal A^{n}_{2, 1/3} 1 \to 0,
\]
implying that $g_{2, 1/3}(x)\equiv 1$ and $V^{\ast}_{2, 1/3}=\infty$
almost surely. The recurrence of $\SFM(2, 1/3)$ thus follows. Once we
establish \eqref{eqn:RDE}, we can compare \eqref{eqn:RDE} with the
$d=2$ case thanks to the recursive algorithm. It turns out that when
choosing $p^{\ast} = \frac{d-1}{2d-1}$, all operators
$\mathcal A_{d, \frac{d-1}{2d-1}}$ are dominated by
$\mathcal A_{2, 1/3}$, yielding $g_{d, \frac{d-1}{2d-1}}(x)\equiv 1$
(see Section \ref{sec:propertyAd}), thus finishing the proof of
Theorem \ref{thm:upbd}.

\section{Proof of Proposition \ref{prop:keycompare}: the coupling\label{sec:proof-prop-coupling}}
In this section, we construct couplings among three types of frog
processes on rooted $d$-ary tree $\mathbb T_d$, namely, the classic
frog model $\FM(d,p)$, the nonbacktracking frog model
$\text{nbFM}(d, p)$, and the self-similar frog model
$\SFM(d,p)$. There are other auxiliary frog processes involved in the
couplings, which we call $\mathcal P_1, \mathcal P_2$ and so
on. Proposition \ref{prop:keycompare} follows from these couplings. We
first give precise descriptions for the nonbacktracking frog model and
the self-similar frog model.

The {\bf non-backtracking frog model} {nbFM$(d, p)$} on the rooted
$d$-ary tree $\mathbb T_{d}$ with drift parameter $p \in [0, 1]$ evolves
according to the same rules as $\FM(d,p)$, with two exceptions. 
First, the paths of active frogs are \emph{non-backtracking}, that is, an
active frog never returns to a site it has previously visited. Second,
active frogs are killed upon visits to the root vertex. Thus, only one
child vertex of the root, which we will henceforth denote by
$\rootvertex'$, can ever be visited in nbFM$(d,p)$: this is the vertex
to which the initially active frog at $\emptyset$ jumps on its first
step. Any other frog, 
upon awakening, will begin its journey with either an upward move to 
the parent vertex with probability $p$ or otherwise a downward move
to a uniformly-chosen child vertex.  In accordance with the 
non-backtracking rule, all following steps must satisfy 
(i) if the last step is upward, the next step will be upward with probability 
$p/(p+(1-p)(d-1)/d)$ or otherwise downward to any one of the child vertices equally likely, 
unless it is at the root $\emptyset$, in which case it is killed; and (ii) if the last step
is downward, the next step will always be downward, equally likely to 
any child vertex.

The {\bf self-similar frog model} $\SFM(d, p)$ on the rooted $d$-ary
tree $\mathbb T_{d}$ with drift parameter $p\in [0, 1]$ can be
constructed by modifying the frog paths in nbFM$(d, p)$ as follows.  
Let $o_{1}, \cdots  ,o_{d}$ be the
child vertices of $\rootvertex'$, the vertex chosen by the initially
active frog at the root on its first jump. Whenever one of these vertices $o_{i}$
is visited for a first time (from $\rootvertex'$), it becomes lethal to
frogs that attempt to jump to it from $\emptyset'$ forever afterward; any
such attempt results in the death of the frog attempting the
jump. Thus, each of the subtrees $\mathbb{T}_{d}(o_{i})$ is entered
from the outside at most once, and conditional on the event that there
is such an entry, the restriction of $\SFM(d,p)$ to this subtree is a
(time-shifted) replica of $\SFM(d,p)$ in the tree
$\mathbb{T}_{d}(\rootvertex')$. This ``self-similarity'' will allow us
to write a ``recursive distributional equation'' or 
``self-consistency equation'' for the generating function of the total
number of  visits to the root: see \cite{HJJ1} and Section \ref{sec: proofofkey}
below.

We next consider a natural embedding of $\mathbb T_d$ in an unrooted
$(d+1)$-ary homogeneous tree $\mathbb T_{d+1}^\text{homo}$ so that
$\mathbb T_{d}$ is isomorphic to a subtree
$\mathbb T'_{d} \subset \mathbb T_{d+1}^{\text{homo}}$. Here, by an
unrooted $(d+1)$-ary homogeneous tree, we mean an infinite tree in
which {\em every} vertex is connected to exactly $(d+1)$ other
vertices (see Figure 1). In this embedding, we associate the root
vertex $\rootvertex$ of $\mathbb T_d$ with with an arbitrary vertex in
$\mathbb T_{d+1}^{\text{Homo}}$ and call it $ \widetilde
\rootvertex$. Fixing an embedding, we assign a level to each vertex in
$\mathbb T_{d+1}^{\text{homo}}$, starting from
$\widetilde \rootvertex \in \mathbb T_{d+1}^{\text{homo}}$ -- the
levels correspond to the ``generations'' in $\mathbb T_{d}$. First,
the level of $\widetilde \rootvertex$ is set to 0 and the level of the
$d$ neighbors of $\widetilde \rootvertex$ in
$\mathbb T_{d+1}^{\text{homo}}$ corresponding to the $d$ child
vertices of $\rootvertex \in \mathbb T_d$ is set to 1. We proceed
until finishing assigning levels for all vertices in the subtree
$\mathbb T'_d$. For other vertices in
$\mathbb T_{d+1}^{\text{homo}} \setminus \mathbb T'_d$, we assign them
levels so that each vertex at level $k$ is connected to $d$ ``child''
vertices in level $(k+1)$ and one ``parent'' vertex in level
$(k-1)$. For example, for $\widetilde \rootvertex$, since it is
already connected to $d$ vertices in level $1$ during the first stage,
the only vertex that has not been assigned a level is then marked a
level $-1$. Figure 1 illustrates how levels $-1$ to 3 would look like
in $\mathbb T_3^{\text{homo}}$.

\begin{figure}[ht]
    \centering
    \includegraphics[width=\textwidth]{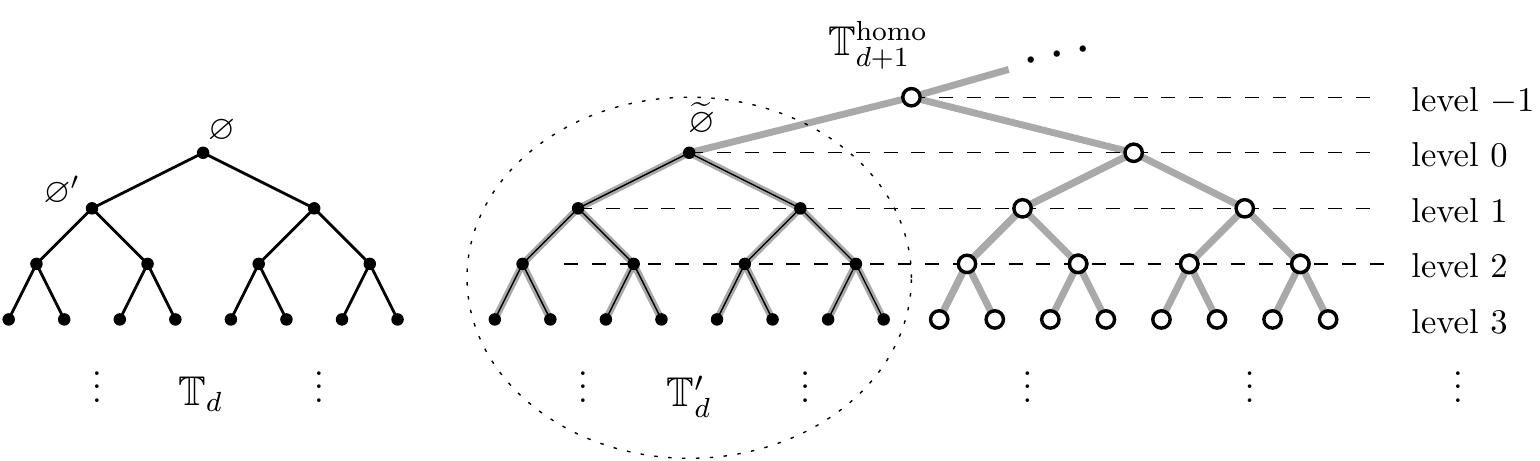}
    \caption{}
    \label{fig:my_label}
\end{figure}

Having determined the subtree
$\mathbb T_d'\subset \mathbb T^{\text{ homo}}_{d+1}$ and the levels,
we run a frog process on $\mathbb T_{d+1}^{\text{homo}}$, and call it
$\mathcal P_1 = \mathcal P_1 (\mathbb T_{d+1}^{\text{homo}}, p)$. At
time $t=0$, a sleeping frog is placed at every vertex of the subtree
$\mathbb T'_d$ (black vertices in Figure 1), and at time $t=1$ the
frog at $\widetilde \rootvertex$ (call it
$f_{\widetilde \rootvertex}$) wakes up and moves to an uniformly
chosen child vertex in level $1$. After the first step, it performs a
$p$-biased random walk on $\mathbb T_{d}'$ {\bf with a reflecting
  boundary at $\widetilde \rootvertex$}, that is, with probability $p$
it moves to the vertex at one level above, with probability $(1-p)$ it
moves to a uniformly-chosen vertex in the level below, and its path is
always reflected at $\widetilde \rootvertex$. As before, whenever an
active frog visits a vertex with a sleeping frog, it wakes up that
frog. Upon waking up, all frogs (except $f_{\widetilde \rootvertex}$)
perform independent $p$-biased random walks on the {\bf entire}
homogeneous tree $\mathbb T_{d+1}^{\text{homo}}$ starting from the
vertices where they originally
sleep. 

From $\mathcal P_1$, we may construct a one-per-site frog process
$\mathcal P_2 = \mathcal P_2(\mathbb T_d, p)$ on $\mathbb T_d$. Since
there is a natural graph isomorphism between $\mathbb T_d$ and
$\mathbb T'_d\subset \mathbb T_{d+1}^{\text{homo}}$, any frog random
walk path $\gamma'$ in $\mathcal P_1$, trimming off the segments spent
outside the subgraph $\mathbb T'_d$, can be translated under the
isomorphism to a $\gamma$ on $\mathbb T_d$ (which might be of finite
length). To this end, we can assign to each frog $f$ in $\mathcal P_2$
the path trimmed and translated from that of the the frog $f'$ in
$\mathcal P_1$ originally placed on the corresponding vertex of
$\mathbb T'_d$. It is not hard to see that $\mathcal P_2$ is dominated
by the frog model $\FM(d, p)$ in the sense that at any time $t$, the
number of active frogs and the total number of visits to the root
vertex in $\mathcal P_2$ in the long run are no more than the
corresponding quantities in $\FM(d, p)$. This is because frog paths in
$\mathcal P_2$ can be considered as those in $\FM(d, p)$ randomly
stopped at the root vertex $\rootvertex$, and such modifications can
only delay waking up frogs and/or reduce the number of visits to the
root vertex $\rootvertex$. To this end, if we let $V_\mathcal P$ be
the total number of visits to the root vertex in a frog process
$\mathcal P$ on $\mathbb T_{d}$, then
$V_{\mathcal P_2(\mathbb T_d, p)} \preceq V_{\FM(d, p)}$.

Now let's consider another one-per-site frog process
$\mathcal P_{3}=\mathcal P_3(\mathbb T_d, p)$ on $\mathbb T_d$ also
constructed from $\mathcal P_1$: for any frog $f^\prime$ in process
$\mathcal P_1$ {\em not} starting from $\widetilde \rootvertex$, if
its $p$-biased random walk path is $\gamma'_{f^\prime}$, we remove all
loops in $\gamma'_{f^\prime}$ first and then terminate the loopless
path at the first visit to $\widetilde \rootvertex$. Under the graph
isomorphism, the resulting path is mapped to a path $\gamma$ on
$\mathbb T_d$ (possibly of finite length) and assigned to the
corresponding frog $f$ in the process $\mathcal P_{3}$; (ii) for the
frog starting from $\widetilde \rootvertex$, we simply remove all
loops in its $p$-biased random walk path, which is then assigned to
the corresponding frog in $\mathcal P_3$ starting from
$\rootvertex \in \mathbb T_d$. We have that
$V_{\mathcal P_{3}(\mathbb T_d, p)} \preceq V_{\mathcal P_2(\mathbb
  T_d, p)}$. To see this, if we drive both $\mathcal P_{2}$ and
$\mathcal P_3$ by the same realization of
$\mathcal P_{1}(\mathbb T_{d}^{\text{Homo}}, p)$, then the paths of
frogs in $\mathcal P_3$ can always be obtained by further trimming the
paths of corresponding frogs in $\mathcal P_2$. Therefore, under this
coupling, if a frog $f_{v}$ sleeping at some vertex
$v\in \mathbb T_{d}$ is ever activated in $\mathcal P_3$, the frog in
$\mathcal P_2$ sleeping at the same vertex of $\mathbb T_{d}$ must
also be activated. Furthermore, since each activated frog in
$\mathcal P_3$ will visit no more sites than its counterpart in
$\mathcal P_2$, the desired stochastic dominance follows. We further
observe that

\begin{lemma}
\label{lemma:matching-prob}
The frog process $\mathcal P_{3}=\mathcal P_{3}(\mathbb T_{d}, p)$ is
a nbFM$(d, p^\ast)$ with $p^\ast$ chosen as in \eqref{eqn:prop2}.
\end{lemma}

With Lemma \ref{lemma:matching-prob}, we establish the hierarchy of
stochastic dominance, namely
\[
V_{\text{SFM}(d, p^\ast)} \preceq V_{\text{nbFM}(d, p^\ast)}=V_{\mathcal P_{3}(\mathbb T_{d}, p)} \preceq V_{\mathcal P_2(\mathbb T_d, p)} \preceq V_{\FM(d, p)},
\]
which implies Proposition \ref{prop:keycompare}.

\begin{proof}[Proof of Lemma \ref{lemma:matching-prob}] In
  $\mathcal P_1$, since $p<1/2$, every active frog will eventually
  drift away to level $\infty$. In particular, after removing all
  loops, the resulting path assigned to the frog originally placed at
  $\rootvertex$ in $\mathcal P_{3}$ is an infinite ray in
  $\mathbb T_d$ chosen uniformly from all possible such rays, the same
  law as that of the frog path in nbFM$(d, p^\ast)$ started at the
  root vertex, because in nbFM$(d, p^\ast)$, the frog started at the
  root will move to a uniformly-chosen child vertex at every step.

  Now we consider frogs in $\mathcal P_3$ that are not placed at the
  root vertex.  Note that for any frog path $\gamma'$ in
  $\mathcal P_1$ started at some vertex $v' \in \mathbb T'_d$ and
  $v^\prime \ne \widetilde \rootvertex$, removing all its loops and
  then terminating it at the first visit to $\widetilde \rootvertex$
  would then always map to a non-backtracking path $\gamma$ on
  $\mathbb T_d$ that (i) starts at the corresponding non-root vertex
  $v\in \mathbb T_{d}$, (ii) first leads up toward the root for $k_1$
  steps (for some $k_1 \le |v|$, where $|v|$ denotes the graph
  distance between $v$ and $\rootvertex$), and then (iii) drifts to
  infinitely far away from the root (when $k_1 < |v|$) or stops at
  $\rootvertex$ (when $k_{1}=|v|$). By symmetry, for any such path,
  the last segment leading directly to infinity from some vertex is an
  infinite ray chosen uniformly from all rays from the aforementioned
  vertex to infinity for both $\mathcal P_{3}$ and
  nbFM$(d, p^{\ast})$. It suffices to show that for all possible
  values of $k_1$, the probability that a frog in $\mathcal P_{3}$
  gets assigned a path of such a pattern is the same as the
  probability that such a pattern is observed in nbFM$(d, p^\ast)$ if
  we choose $p^\ast$ as in \eqref{eqn:prop2}. There are three cases.
\begin{enumerate}[(a)]
\item $k_1=0$. In nbFM$(d, p^\ast)$, a frog at a non-root vertex
  immediately moves away from the root upon waking up with probability
  $(1-p^\ast)$. In $\mathcal P_3$, such a non-backtracking path can be
  obtained by the loop-erasural procedure from a frog path in
  $\mathcal P_1$ with probability
\[
\sum_{l=0}^\infty \rho^l (1-\rho) \left(\frac{1}{d}\right)^l = \frac{1-\rho}{1-\rho/d} = 1 - \frac{\rho(1-1/d)}{1-\rho/d} = 1-p^\ast,
\]
where $\rho := \frac{p}{1-p}$ is the probability that a random walk on
$\mathbb Z$ starting from 0 with step distribution
$p\delta_{-1} + (1-p)\delta_{+1}$ never visits location $-1$. In the
summation, the variable $l$ tracks the furthest distance that a frog
in $\mathcal P_1$ has ever reached above its sleeping level. After
that, it must trace backward along the same route, return to the
vertex it starts from and then drift to level $\infty$, because only
in this way, this loop will be removed and we are left with a
loop-erased path $\gamma$ with the desired pattern.
\item $k_1=1,\ldots, |v|-1$ when $|v|\ge 2$. In nbFM$(d, p^\ast)$, the probability that a frog at a non-root vertex $v$ first moves $k_1 < |v|$ steps toward the root and then moves away to infinity is  
$$
p^\ast \, \left(\frac{p^\ast}{\frac{d-1}{d}(1-p^\ast)+p^\ast}\right)^{k_1-1} \left(\frac{\frac{d-1}{d}(1-p^\ast)}{\frac{d-1}{d}(1-p^\ast)+p^\ast}\right).
$$
In the frog process $\mathcal P_3$, such a non-backtracking path can be obtained from a frog path in $\mathcal P_1$ with probability 
$$
\sum_{l=0}^\infty \rho^{k_1+l}(1-\rho) \left(\frac{1}{d}\right)^l \left(\frac{d-1}{d}\right) =  \frac{\rho^{k_1} (1-\rho)(1-1/d)}{1-\rho/d}.
$$
Similar to the first case, $k_1+l$ denotes the furthest distance that
a frog in $\mathcal P_1$ has ever reached above its sleeping level. If
after removing all loops, there are still $k_1$ upward steps left in
the resulting non-backtracking path, then the frog in $\mathcal P_{1}$
must have travelled exactly $l$ steps along the same route that had
led it $(k_{1}+l)$ levels up.  By choosing $p^\ast$ as in
\eqref{eqn:prop2}, the above two expressions are equal.
\item $k_1=|v|$. In nbFM$(d, p^\ast)$, the probability that a frog at a non-root vertex moves straight to the root vertex upon waking up is 
\[
p^\ast \, \left(\frac{p^\ast}{\frac{d-1}{d}(1-p^\ast)+p^\ast}\right)^{|v|-1},
\]
whereas a path of the same pattern can be obtained by trimming a frog
path in $\mathcal P_1$ with probability
\[
\sum_{l=0}^\infty \rho^{|v|+l} (1-\rho) \sum_{m=0}^{l} \left(\frac{1}{d}\right)^m\left(\frac{d-1}{d}\right). 
\]
In the above expression, $l$ marks the number of levels above
$\widetilde \rootvertex$ that the frog in $\mathcal P_{1}$ has ever
reached and $m$ denotes the number of steps that the frog has traced
back. It is again easy to verify that these two expressions are equal
when $p^\ast$ is chosen as in \eqref{eqn:prop2}.
\end{enumerate}
The proof of Lemma \ref{lemma:matching-prob} is complete.
\end{proof}

\section{Proof of Proposition \ref{prop:key}: the self-consistency equation }
\label{sec: proofofkey}
In this section and the next, we will prove that the self-similar frog
model $\SFM(d, \frac{d-1}{2d-1})$ is recurrent. The proof has two
steps. The first step, to which this section is devoted, will be
to establish a fixed-point equation
\begin{equation}
\label{eqn:self-consist}
g_{d, p}(x) = \mathcal A_{d, p} g_{d, p}(x)
\end{equation}
for the probability
generating function $g_{d,p}(x) := \E(x^{V^\ast_{d,p}})$ of the number
$V^\ast_{d, p}$ of visits to the root vertex $\rootvertex$. The
second, which will be carried out in Section~\ref{sec:propertyAd},
will be to use the fixed-point equation to show that $g_{d,p}(x)
=0$ for all $x \in (0,1)$; it will then follow that $V^\ast_{d,
  \frac{d-1}{2d-1}} = \infty$ almost surely, proving that $\SFM(d,
\frac{d-1}{2d-1})$ is recurrent. This strategy was  used
in \cite{HJJ1} to show that $\SFM(2, 1/3)$ is recurrent and in
\cite{josh_32} to show that the frog model on a $(2,3)$-alternating
tree is recurrent.

The fixed-point equation \eqref{eqn:self-consist} involves a
nonlinear operator $\mathcal{A}_{d,p}$ whose domain is the function space
\begin{displaymath}
  \mathcal I := \{f:[0,1)\to [0,1], \text{ nondecreasing}\},
\end{displaymath}
which contains the generating function $g_{d,p}(x)$ as an element.
This operator is a polynomial in composition operators
$Z_{k}:\mathcal{I}\rightarrow \mathcal{I}$
defined  for any $p \in (0,1)$ and
$k=1,\cdots, d  $ by
\begin{align}
\label{eqn:def-zk}
Z_{k}(h)&:=Z_{k,d,p}(h)=h\circ c_{d, p}^{(k-1)}, \\
\text{where }\ \ 
\label{eqn:constC}
c_{d,p}^{(k-1)}(x)&:=\frac{px + (1-p)(k-1)/d}{p+(1-p)(d-1)/d} \quad \text{for } x\in [0, 1]. 
\end{align}

Observe that the functions $c_{d,p}^{(k)}$ where $k=0, 1, \ldots, (d-1)$, are linear functions that map
the unit interval monotonically onto subintervals of $[0,1]$. The operator $\mathcal{A}_{d,p}$ is defined   as follows: for any $h \in \mathcal{I}$,
\begin{align}
\notag
  \mathcal{A}_{d,p} h (x)&= \left(px + \frac{1-p}{d}\right) \sum_{k=1}^{d}\binom{d-1}{k-1}P_{k}(Z_{1}(h)(x),
  \ldots, Z_{k}(h)(x) ) \\
    \label{eq:Adp-def}
  & \qquad + \frac{(d-1)(1-p)}{d} \sum_{k=2}^{d}\binom{d-2}{k-2}Q_{k}(Z_{1}(h)(x),\ldots,
  Z_{k}(h)(x) ),
\end{align}
where $P_{k}: \mathbb R^{d} \to \mathbb R$ and  $Q_{k}:\mathbb R^{k}\to
\mathbb R$ are polynomials in $k$ variables defined  recursively  by
\begin{align}
\label{eqn:p-recursion}
    P_{k+1} (z_1, \ldots, z_{k+1}) &=z_{k+1}^{k+1} - \sum_{l=1}^{k}
                                     \binom{k}{l-1}z_{k+1}^{k+1-l}P_{l}(z_1,\ldots,
                                     z_l);\\ 
 \label{eqn:q-recursion}
    Q_{k+1} (z_1, \ldots, z_{k+1}) &=z_{k+1}^{k+1} - \sum_{l=2}^{k}
                                     \binom{k-1}{l-2}z_{k+1}^{k+1-l}Q_{l}(z_1,\ldots,
                                     z_k),  
\end{align}
starting with $P_{1}(z_{1})=z_{1}$ and $Q_{2}(z_{1}, z_{2}) = z_{2}^{2}$.

\begin{prop} \label{prop:recursive} For any $p \le 1/2$ and $d\ge 2$,
    the generating function $g_{d,p}(x) := \E(x^{V^\ast_{d,p}})$ 
  for the total number of visits to the root in $\SFM(d, p)$ satisfies
  equation \eqref{eqn:self-consist}.
\end{prop}
\begin{remark}   The self-consistency equation \eqref{eqn:self-consist} for
  $g_{d+1,p}(x)$ involves the same polynomials $P_{k},Q_{k}$ as that
  for $g_{d,p}$, but evaluated at the variables 
  \begin{align*}
  (z_{1}, \ldots z_{k}) &= \big (Z_{1, d+1, p}(g_{d+1, p}) (x), \ldots, Z_{k, d+1, p}(g_{d+1, p}) (x)\big )\\ &=\big( g_{d+1, p}\circ c_{d+1, p}^{(0)}(x), \ldots, g_{d+1, p}\circ c_{d+1, p}^{(k-1)}(x) \big)
  \end{align*}
in place of  
\begin{align*}
  (z_{1}, \ldots z_{k}) &= \big (Z_{1, d, p}(g_{d, p}) (x), \ldots, Z_{k, d, p}(g_{d, p}) (x)\big ) \hspace{1.4cm}\\
  &=\big( g_{d, p}\circ c_{d, p}^{(0)}(x), \ldots, g_{d, p}\circ c_{d, p}^{(k-1)}(x) \big)
  \end{align*}
  The additional terms $k=d+1$ in the sums
  defining $\mathcal{A}_{d+1,p}$ involve only the polynomials
  $P_{d+1},Q_{d+1}$, which, by equations \eqref{eqn:p-recursion} and
  \eqref{eqn:q-recursion}, are themselves defined in terms of the
  polynomials $P_{k},Q_{k}$ that appear in the definition of
  $\mathcal{A}_{d,p}$.  This suggests that the equations
  \eqref{eqn:self-consist} for the generating functions $g_{d,p}$
  can be established recursively, given the ``base'' case  $\mathcal
  A_{2,p}$.
\end{remark}

\begin{remark} The definition of the operator $\mathcal A_{d,p}$ does not ensure automatically that $\mathcal A_{d,p}h \in \mathcal I$ for every $h \in \mathcal I$. However, we will only need to be able to apply $\mathcal A_{d, p}$ to $g_{d,p}$ repeatedly, which is implied by the fixed point equation \eqref{eqn:self-consist}.
\end{remark}

\subsection{The self-similar structure} We begin by describing the self-similar
structure in $\SFM(d, p)$, which explains why $Z_{k, d, p}(g_{d, p})$'s, as define in \eqref{eqn:def-zk}, appear in the self-consistency equation. To simplify the notation, when there is no
ambiguity we omit the subscript $d,p$ and only write
$$g(x)=g_{d,p}(x) = \E(x^{V^{\ast}}) = \mathcal A g(x), \ \text{where
}V^{\ast}=V^{\ast}_{d,p}. $$
Recall that, in $\SFM$,
only one of the $d$ child vertices of the root $\varnothing$ could be
visited, and we call this vertex $\rootvertex'$ and denote its child
vertices by $o_{1}, \ldots, o_{d}$ as before. The $d$-ary subtree
rooted at any vertex $u$ is denoted by $\mathbb T_d(u)$.  For
$j = \rootvertex$ or $\rootvertex'$ define
\begin{align}
\label{eqn:defV-itoj}
    V^{\ast}_{i\to j} := \text{number of visits to vertex $j$ by all frogs originally placed in $\mathbb T_d(o_i)$}.
\end{align}
Denote by $f_{v}$ the frog initially placed at a vertex
$v\in \mathbb T_{d}$ and $\{f_{v}\to u\}$ the event that the frog
$f_{v}$ has visited the vertex $u$. Since only one frog is ever allowed to enter $\mathbb T_d(o_i)$ in
$\SFM$, conditional on the event that $\mathbb T_{d}(o_{i})$ has been
visited, $V^\ast_{i \to \rootvertex'}$ would have the same
distribution as $V^\ast$, which implies
\begin{equation}
\label{eqn:self-similar}
P(V^{\ast}=n) = P(V^\ast_{i\to \rootvertex'}=n|\mathbb T_{d}(o_{i}) \text{ is ever visited}).
\end{equation}
\begin{figure}[ht]
    \centering
    \includegraphics{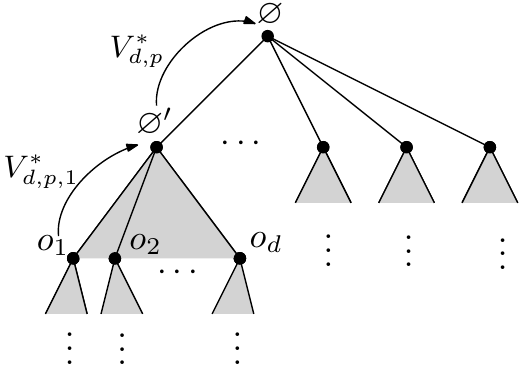}
    \caption{The self-similar structure in $\SFM(d,p)$.}
    \label{fig:self-similar}
\end{figure}
Since all frogs behave independently, given that
$V_{i\to \rootvertex' }^{\ast}=n\ge 1$, the number of visits from
$o_{i}$ to the root vertex $\rootvertex$ in $\SFM(d, p)$ via
$\rootvertex'$, $V^{\ast}_{i\to \rootvertex}$, follows a
Bin$\left(n, \frac{p}{p+(1-p)\frac{d-1}{d}}\right)$ distribution, and
thus the generating function satisfies
\begin{align}
\label{eqn:ss-binomial}
\E(x^{V^{\ast}_{i\to \rootvertex}} | V^{\ast}_{i\to \rootvertex'}=n) =\left(\frac{px}{p+(1-p)\frac{d-1}{d}}+\frac{(1-p)\frac{d-1}{d}}{p+(1-p)\frac{d-1}{d}}\right)^{n}. 
\end{align}
Lemma \ref{lem:binomial} is a generalization of the observation. For
any child vertex $o_{j}$ of $\rootvertex'$ and
$J\subseteq \{o_1,\ldots, o_d\}\setminus\{o_{j}\}$, define the event
\begin{align}
\label{eqn:eventB}
B_d(o_{j}; J) &:=\bigcap_{s\in J}\left\{\begin{array}{l}
\text{frogs initially placed on vertices of }\\ 
 \mathbb T_{d}(o_{j}) \text{ never enter subtrees } \mathbb T_{d}(s)
\end{array}\right\}
\end{align}
and if $J$ is the empty set, then $B_d(o_j; J)=\Omega$, the entire
probability space. We note that, if $|J|=k$, due to symmetry, the
probability that frogs in $\mathbb T_d(o_{j})$ never enter any $k$
other subtrees specified by $J$ is the same for all possible choices
of $J$.  
\begin{lemma} \label{lem:binomial} For any $d\ge 2$, $p\in [0, 1/2]$,
  $j \in [d]$, and $J\subseteq \{o_1,\ldots, o_d\}\setminus\{o_j\}$,
  we have
\[
\E (x^{V^{\ast}_{j\to \rootvertex}}\mathbf 1_{B_{d}(o_j; J)}|\mathbb T_{d}(o_{j}) \text{ is ever visited}) = g\circ c^{(d-1-|J|)}(x).
\]
where $c^{(k)}(x) = c_{d, p}^{(k)}(x)$ is as defined in \eqref{eqn:constC}, for $k=0, 1,\ldots, (d-1)$.
\end{lemma}
\begin{proof} We compute the expectation by conditioning on the number
  of visits from $o_j$ to $\rootvertex'$.
\begin{align*}
  &\E (x^{V^{\ast}_{j\to \rootvertex}}\mathbf 1_{B_{d}(o_j; J)}|\mathbb T_{d}(o_{j}) \text{ is ever visited}) \\
  &=\E [\E(x^{V^{\ast}_{j\to \rootvertex}}\mathbf 1_{B_{d}(o_j; J)} | V^{\ast}_{j\to \rootvertex'},\ \mathbb T_{d}(o_{j}) \text{ is ever visited})]\\
  &=\sum_{n=0}^{\infty} \sum_{k=0}^{n}x^{k}P (V^{\ast}_{j\to \rootvertex}=k \text{ and } B_{d}(o_j; J)|V_{j\to \rootvertex'}^\ast=n, \mathbb T_{d}(o_{j}) \text{ is ever visited})\\
  &\qquad \cdot P(V_{j\to\rootvertex'}^{\ast}=n|\mathbb T_{d}(o_{j}) \text{ is ever visited}).
\end{align*}
By \eqref{eqn:self-similar}, the last probability is equal to
$P(V^{\ast}=n)$. Moreover, on the event
$V_{j\to \rootvertex'} ^{\ast}= n \ge 1$, exactly $k$ of the $n$ frogs
that visit $\rootvertex'$ from $o_{j}$ would continue to move up to
the root $\rootvertex$ with probability
$$\binom{n}{k} \left[\frac{p}{p+(1-p)(d-1)/d}\right]^{k} \left[1-\frac{p}{p+(1-p)(d-1)/d}\right]^{n-k},$$
and, on $B_{d}(o_{j}; J)$, those $(n-k)$ frogs that do not move to the
root $\rootvertex$ must all avoid going to any vertex $s\in J$ with
probability $ \left[\frac{(d-1)-|J|}{d-1}\right]^{n-k} $. If
$V_{j\to \rootvertex'} ^{\ast}= 0$, then $V^{\ast}_{j\to \rootvertex}=0$
with probability one. We thus get
\begin{align*}
&\E (x^{V^{\ast}_{j\to \rootvertex}}\mathbf 1_{B_{d}(o_j; J)}|\mathbb T_{d}(o_{j}) \text{ is ever visited})\\
&=\sum_{n=0}^{\infty} \sum_{k=0}^{n}x^{k} \binom{n}{k} \left[\frac{p}{p+(1-p)(d-1)/d}\right]^{k} \left[1-\frac{p}{p+(1-p)(d-1)/d}\right]^{n-k}\\
&\qquad \cdot  \left[\frac{(d-1)-|J|}{d-1}\right]^{n-k} P(V^{\ast}=n)\\
&=\sum_{n=0}^{\infty}\left\{ \sum_{k=0}^{n}\binom{n}{k}\left(\frac{px}{p+(d-1)(1-p)/d}\right)^{k} \left(\frac{\frac{(d-1)(1-p)}{d}  \cdot \frac{d-1-|J|}{d-1}}{p+(d-1)(1-p)/d}\right)^{n-k}\right\}P(V^{\ast}=n)\\
&=\sum_{n=0}^{\infty}\left[\frac{dpx + (d-1-|J|)(1-p)}{dp + (d-1)(1-p)}\right]^{n}P(V^{\ast}=n)\\
&=\sum_{n=0}^{\infty} [c^{(d-1-|J|)}(x)]^{n}P(V^{\ast}=n)=\E \{ \,[c^{(d-1-|J|)}(x)]^{V^{\ast}}\} = g\circ c^{(d-1-|J|)}(x).
\end{align*}
\end{proof}
\begin{remark} When $J$ is the empty set, frogs initially placed at
  vertices of $\mathbb T_{d}(o_{j})$ may go to any of the other
  $(d-1)$ subtrees (i.e., no constraints) and thus
  $\mathbf 1_{B_{d}}(o_{j}; J)=\mathbf 1$. We have
\begin{align}
\label{eqn:binomialsp}
\E (x^{V^{\ast}_{j\to \rootvertex}}|\mathbb T_{d}(o_{j}) \text{ is ever visited}) = g\circ c^{(d-1)}(x).
\end{align}
\end{remark}
\begin{remark} The composition of  $g$ with $c^{(k-1)}$ is exactly the generating function for the number of visits to the root by frogs in a subtree $\mathbb T_{d}(o_{j})$ (if the branch is activated) on the event that none of these frogs ever go to some other $(k-1)$ subtrees.
\end{remark}
\subsection{The $P$- and $Q$- polynomials}
For the clarity of the recursion, we now put back the subscript $d$ but still
omit the parameter $p$, since $p$ is unchanged.  In this section, we explain what the $P$- and $Q$-polynomials represent in the self consistency equation \eqref{eq:Adp-def}.
We begin by decomposing $g_d(x)$
according to the number of child vertices of $\rootvertex'$ that have
been ever visited. Write
\begin{align}
\notag
g_d(x) &= \sum_{k=1}^{d}\sum_{\substack{S: |S|=k, \\ S\subseteq [d]}}\E(x^{V_d^{\ast}}\mathbf1_{\cap_{i\in S}\{\mathbb T_{d}(o_{i})\text{ is ever visited} \}}\mathbf1_{\cap_{j\in S^{c}}\{\mathbb T_{d}(o_{j})\text{ is not visited} \}})\\
\label{eqn:g-1stdecomp}
& = \sum_{k=1}^{d} \binom{d}{k} k \E(x^{V_{d}^{\ast}}\mathbf 1_{\{f_{\rootvertex} \to o_{1}\}} \mathbf 1_{A_{d,k}})
\end{align}
where 
\begin{align*}
A_{d,k}&:= \bigcap_{i=1}^{k}\{\mathbb T_{d}(o_{i})\text{ is visited} \}\bigcap_{j=k+1}^{d}\{\mathbb T_{d}(o_{j})\text{ is not visited} \}.
\end{align*}
The factors $\binom{d}{k}$ and $k$ in \eqref{eqn:g-1stdecomp} 
are due to symmetry: (i) any
subset $S\subseteq [d]$ of size $k$ would contribute to $g_d(x)$ in
the same way and there are $\binom{d}{k}$ choices for $S$, and (ii) on
the event that exactly
$\mathbb T_{d}(o_{1}),\ldots, \mathbb T_{d}(o_{k})$ are ever visited,
the frog $f_{\rootvertex}$ originating from the root vertex  is equally
like to visit any one of them, resulting in the additional factor $k$. 

We further decompose each summand into two cases, depending on whether or
not $f_{\rootvertex'}$ activates a new subtree
$\mathbb T_{d}(o_{i})\ne \mathbb T_{d}(o_{1})$.  Define events
\begin{align*}
D_1&:=\{f_{\rootvertex'} \to o_{1}\} \cup\{f_{\rootvertex'}\to \rootvertex\}, \quad D_2 := D_1^{c} =  \bigcup_{i=2}^{d}\{f_{\rootvertex'} \to o_{i}\}.
\end{align*}
Note that on $D_2$, at least two branches of $\rootvertex'$ are visited,  the intersection $D_2 \cap A_{d, 1}$ is empty.
We then express $g_{d}(x)$ as two summations of $(2d-1)$ terms in total.
\begin{align}
\notag
g_{d}(x) & =\sum_{k=1}^{d} \binom{d}{k} k \big[\E(x^{V_{d}^{\ast}}\mathbf 1_{\{f_{\rootvertex} \to o_{1}\}} \mathbf 1_{D_1\cap A_{d,k}}) +  \E(x^{V_{d}^{\ast}}\mathbf 1_{\{f_{\rootvertex} \to o_{1}\}} \mathbf 1_{D_2\cap A_{d,k}})\big]\\
\label{eqn:dk1}
&= \sum_{k=1}^{d} \binom{d}{k} k \E(x^{V_{d}^{\ast}}\mathbf 1_{\{f_{\rootvertex} \to o_{1}\}} \mathbf 1_{D_1\cap A_{d,k}}) \\
\label{eqn:dk2}
&\qquad + \sum_{k=2}^{d} \binom{d}{k} k \E(x^{V_{d}^{\ast}}\mathbf 1_{\{f_{\rootvertex} \to o_{1}\}} \mathbf 1_{D_2 \cap A_{d,k}}).
\end{align} 
On $D_1$, since $f_{\rootvertex'}$ goes to either $o_1$ (with
probability $(1-p)/d$) or the root vertex $\rootvertex$ (with
probability $p$, which contributes one visit to the root), then
$V_d^\ast$ can be expressed as
\[
V_d^\ast = \mathbf 1_{\{f_{\rootvertex'}\to \rootvertex\}} + \sum_{j=1}^d V^\ast_{d, j\to \rootvertex}.
\]
Here the quantities $V^\ast_{d, j\to \rootvertex}$ are the same as $V^\ast_{j \to \rootvertex}$ in \eqref{eqn:defV-itoj} except that we now put back the subscript $d$ to emphasize that we are working with the $\SFM(d, p)$ case. 
By independence of $f_\rootvertex$, $f_\rootvertex'$ and frogs in $\mathbb T_d(o_i)$, each summand in \eqref{eqn:dk1} can be written as
\begin{align*}
&\qquad \binom{d}{k}\,  k\,  \E(x^{V_{d}^{\ast}}\mathbf 1_{\{f_{\rootvertex} \to o_{1}\}} \mathbf 1_{D_1\cap A_{d,k}}) \\
&=\binom{d}{k}\, k\, P(f_{\rootvertex}\to o_{1}) \E(x^{\mathbf 1_{\{f_{\rootvertex'} \to \rootvertex \}}}\mathbf 1_{D_{1}}) \E\left(x^{\sum_{j=1}^k V_{d, j\to \rootvertex}^\ast}\mathbf 1_{A_{d, k}} | f_\rootvertex \to o_1, D_1\right)
\\
&=\left(px+\frac{1-p}{d}\right)\binom{d-1}{k-1}\E\left(x^{\sum_{j=1}^k V_{d, j\to \rootvertex}^\ast}\mathbf 1_{A_{d, k}} | f_\rootvertex \to o_1, D_1\right). 
\end{align*}
Similarly, each summand in \eqref{eqn:dk2} can be written as
\[
\binom{d}{k} k  \E(x^{V_{d}^{\ast}}\mathbf 1_{\{f_{\rootvertex} \to o_{1}\}} \mathbf 1_{D_2\cap A_{d,k}}) = \frac{(d-1)(1-p)}{d}\binom{d-2}{k-2}\E\left(\left.x^{\sum_{j=1}^k V_{d,j\to \rootvertex}^\ast}\mathbf 1_{A_{d, k}} \right| \begin{array}{c}  f_\rootvertex \to o_1\\
 f_{\rootvertex'} \to o_2
 \end{array}\right),
\]
where we used fact that on $A_{d,k}$, if $f_{\rootvertex'}$ activates
a new branch other than $\mathbb T_{d}(o_{1})$, it is equally likely
to go to any of the other $(k-1)$ vertices $o_2, \ldots, o_k$.  Denote
these conditional expectations by
\begin{align}
\label{eqn:def-Pdk}
\mathcal P_{d,k}(x)&:= \E\left(x^{\sum_{j=1}^k V_{d,j\to \rootvertex}^\ast}\mathbf 1_{A_{d, k}} | f_\rootvertex \to o_1, D_1\right)\\
\label{eqn:def-Qdk}
\mathcal Q_{d,k}(x) &:= \E\left(\left.x^{\sum_{j=1}^k V_{d, j\to \rootvertex}^\ast}\mathbf 1_{A_{d, k}} \right| \begin{array}{c}  f_\rootvertex \to o_1\\
 f_{\rootvertex'} \to o_2
 \end{array}\right).
\end{align}
and we have
\begin{align}
\label{eqn:gd-final}
g_{d}(x) = \left(px + \frac{1-p}{d}\right)\sum_{k=1}^{d}\binom{d-1}{k-1} \mathcal P_{d,k}(x) + \frac{(d-1)(1-p)}{d}\sum_{k=2}^{d}\binom{d-2}{k-2}\mathcal Q_{d, k}(x).
\end{align}

Comparing \eqref{eqn:gd-final} with \eqref{eq:Adp-def}, to show that $g_{d}$ is a fixed point of $\mathcal A_{d, p}$,  we require that the polynomials $P_{k}: \mathbb R^{k}\to \mathbb R$ and $Q_{k}: \mathbb R^{k} \to \mathbb R$, constructed recursively through equations \eqref{eqn:p-recursion} and \eqref{eqn:q-recursion}, to satisfy
\begin{align}
\label{eqn:relation1}
    \mathcal P_{d, k}(x) &= P_{k}(g_d\circ c_d^{(0)}(x),\ldots, g_d\circ c_d^{(k-1)}(x)),\\
\label{eqn:relation2}
    \mathcal Q_{d, k}(x) &= Q_{k}(g_d\circ c_d^{(0)}(x), \ldots, g_d\circ c_d^{(k-1)}(x)).
\end{align}
We point out that these $P$- and $Q$-polynomials in the $z$ variables would not depend on the degree  $d$ of the tree, and the dependence of the conditional expectations \eqref{eqn:def-Pdk} and \eqref{eqn:def-Qdk} on $d$ is reflected only when plugging in $z_i = g_d \circ c_d^{(i-1)}(x)$.  Consequently, the first $(2d-1)$ terms in the self-consistence equation for $g_{d+1}(x)$ would share the same ``structures'' as the $(2d-1)$ terms for $g_d(x)$.

\begin{lemma}\label{lem:first-d-terms}
For each $d\ge 2$ and $k\le d$, $\mathcal P_{d,k}(x)$ and $\mathcal Q_{d,k}(x)$, as defined in \eqref{eqn:def-Pdk} and \eqref{eqn:def-Qdk}, respectively, are polynomials of $g_{d}\circ c_{d}^{(0)}(x), \ldots g_{d}\circ c_{d}^{(k-1)}(x)$, where $g_{d}(x):=\E(x^{V^{\ast}_{d}})$ and $c_{d}^{(i)}:[0,1]\to [0,1]$ are linear functions defined in \eqref{eqn:constC}. Moreover, if there are polynomials $P_k: \mathbb R^k \to \mathbb R$ and $Q_k: \mathbb R^k \to \mathbb R$ such that \eqref{eqn:relation1} and \eqref{eqn:relation2} hold for some $d \ge k$,
then 
\begin{align*}
    \mathcal P_{d+1, k} (x) &= P_{k}(g_{d+1}\circ c_{d+1}^{(0)}(x),\ \ldots, g_{d+1}\circ c_{d+1}^{(k-1)}(x))\\
    \mathcal Q_{d+1, k} (x) &= Q_{k}(g_{d+1}\circ c_{d+1}^{(0)}(x),\ \ldots, g_{d+1}\circ c_{d+1}^{(k-1)}(x)).
\end{align*}
\end{lemma}
We defer the proof of Lemma \ref{lem:first-d-terms} to the end of the section.

\subsection{The base case.} 
We explain in detail how we derive the self-consistency equation in the $\SFM(2, p)$
case. We then give an example to show how to construct the self-consistency equation in the $\SFM(3, p)$ case using Lemma \ref{lem:first-d-terms} and equations \eqref{eqn:p-recursion} and \eqref{eqn:q-recursion}.

According to \eqref{eqn:gd-final}, the generating function $g_{2}(x) $ can be written as
\begin{align}
\label{eqn:SFM2case}
g_2(x) =\left(px+\frac{1-p}{2}\right)\left[\mathcal P_{2,1}(x) + \mathcal P_{2, 2}(x)\right] + \frac{(1-p)}{2}\mathcal Q_{2,2}(x).
\end{align}
By definition \eqref{eqn:def-Pdk}, the first term 
\begin{align*}
    \mathcal P_{2,1}(x)&= \E\left(x^{V^\ast_{2, 1\to \rootvertex}}\mathbf 1_{A_{2,1}}|f_\rootvertex \to o_1, D_1\right)\\
    &=\E\left(x^{V^\ast_{2,1\to \rootvertex}}\mathbf 1_{B_2(o_1;\{ o_2\})}|f_\rootvertex \to o_1, D_1\right)\\
    &=\E\left(x^{V^\ast_{2,1\to \rootvertex}}\mathbf 1_{B_2(o_1;\{ o_2\})}|\mathbb T_{2}(o_{1})\text{ is visited}\right) \\
    &= g_2\circ c_2^{(0)}(x),\ \  \text{by Lemma \ref{lem:binomial}}.
\end{align*}      
The second equality above is because on the event
$\{f_\rootvertex \to o_1\}\cap D_1$, the event $A_{2,1}$ is equivalent to that ``no frogs in $\mathbb T_{2}(o_{1})$ ever enter the subtree $\mathbb T_{2}(o_{2})$'', which is exactly 
$B_2(o_1; \{o_2\})$,  defined in \eqref{eqn:eventB}. 
The third equality is due to the independence of frogs in the
subtree $\mathbb T_2(o_1)$ and the frog
$f_{\rootvertex'}$, that is, the number of visits to the root
vertex by frogs in $\mathbb T_{2}(o_{1})$ only depends on whether or
not $\mathbb T_{2}(o_{1})$ is ever visited but not where $f_{\rootvertex'}$ goes.

For $\mathcal Q_{2,2}(x)$, on the event
$\{f_\rootvertex \to o_1\}\cap \{f_{\rootvertex'} \to o_2\}$, we have
$\mathbf 1_{A_{2,2}}=1$. By independence of frogs initially placed in
subtree $\mathbb T_2(o_1)$ and those in $\mathbb T_2(o_2)$,
\begin{align*}
    \mathcal Q_{2,2}(x)&= \E\left(x^{V^\ast_{2,1\to \rootvertex}+V^\ast_{2, 2\to \rootvertex}}\mathbf 1_{A_{2,2}}|f_\rootvertex \to o_1, f_{\rootvertex'} \to o_2\right)\\
    &=\E\left(x^{V^\ast_{2,1\to \rootvertex}}|\mathbb T_{2}(o_{1})\text{ is visited} \right) \E\left( x^{V^\ast_{2,2\to \rootvertex}}|\mathbb T_{2}(o_{2})\text{ is visited}\right)\\
    & = [g_2\circ c_2^{(1)}(x)]^2, \ \  \text{by Lemma \ref{lem:binomial}}.
\end{align*}
Finally, conditional on $\{f_\rootvertex \to o_1\}\cap D_1$, the event
$A_{2,2}$ is the same as that ``some frog initially placed in the
subtree $\mathbb T_2(o_1)$ has entered $\mathbb T_{2}(o_2)$''. Denote
this event by $G$, and we have $G^c = B_2(o_1; \{o_2\})$ on the event
$\{f_\rootvertex \to o_1\}\cap D_1$. Knowing the subtree
$\mathbb T_2(o_2)$ has been visited, the number of visits to the root
vertex by frogs in $\mathbb T_2(o_2)$ does not depend on the behavior
of the frogs in $\mathbb T_2(o_1)$, $f_\rootvertex$ or
$f_\rootvertex'$.  This gives
\begin{align*}
    \mathcal P_{2,2}(x) &= \E(x^{V^\ast_{2, 1\to \rootvertex}+V^\ast_{2, 2\to \rootvertex}}\mathbf 1_{A_{2,2}}|f_\rootvertex \to o_1, D_1)\\
    &=\E(x^{V^\ast_{2,1\to \rootvertex}+V^\ast_{2,2\to \rootvertex}}\mathbf 1_{G}|f_\rootvertex \to o_1)\\
    &=\E(x^{V^\ast_{2,1\to \rootvertex}}\mathbf 1_{G}|f_\rootvertex \to o_1)\, \E(x^{V^\ast_{2,2\to \rootvertex}}|G)\\
    &=\E(x^{V^\ast_{2,1\to \rootvertex}}(\mathbf 1 - \mathbf 1_{B_2(o_1; \{o_2\})})|f_\rootvertex \to o_1)\,\E(x^{V^\ast_{2,2\to \rootvertex}}|\mathbb T_2(o_2)\text{ is activated})\\
    &=g_2\circ c_2^{(1)} (x) [g_2\circ c_2^{(1)} (x) - g_2\circ c_2^{(0)} (x)], \ \  \text{by Lemma \ref{lem:binomial}}.
\end{align*}
Combining all computations above, we get
\begin{align*}
g_2(x) &=   \left [ px+\frac{1-p}{2}\right] \left\{ g_2\circ c_2^{(0)}(x) + g_2\circ c_2^{(1)}(x)\left[g_2\circ c_2^{(1)}(x) - g_2\circ c_2^{(0)}(x)\right]\right\} + \frac{1-p}{2} [g_2\circ c_2^{(1)}(x)]^2.
\end{align*}
Thus, if we define polynomials 
\begin{align*}
    P_{1}(z_1)&:=z_1\\ 
    P_{2}(z_1, z_2)&:=z_2(z_2-z_1)\\
    Q_{2}(z_1, z_2)&:=z_2^2,
\end{align*}
then we obtain
\begin{align*}
    \mathcal P_{2,1}(x)&=P_{1}(g_2\circ c_2^{(0)}(x))\\ 
    \mathcal P_{2,2}(x)&=P_{2}(g_2\circ c_2^{(0)}(x),\ g_2\circ c_2^{(1)}(x))\\
    \mathcal Q_{2,2}(x)&=Q_{2}(g_2\circ c_2^{(0)}(x),\ g_2\circ c_2^{(1)}(x)).
\end{align*}

\begin{remark}
With $p=1/3$, we have $\frac{1-p}{2}=\frac{1}{3}$, $px + \frac{1-p}{2}=\frac{x+1}{3}$, $c_2^{(0)}(x)=x/2$ and $c_2^{(1)}(x) = (x+1)/2$, and the above recovers Equation (1) in \cite{HJJ1} derived for $\SFM(2, 1/3)$:
\begin{align}
    \label{eqn:sfm2-eqn}
g_2 (x)  = \frac{x+2}{3}g_2\left(\frac{x+1}{2}\right)^2 + \frac{x+1}{3}g_2\left(\frac{x}{2}\right)\left[1-g_2\left(\frac{x+1}{2}\right)\right].
\end{align}
\end{remark}

\begin{example} Find the self-consistency equation for $g_3(x)=\E(x^{V^{\ast}_{3}})$ for $\SFM(3, p)$. \\ 
First, according to \eqref{eqn:gd-final}, $g_3(x)$ can be written as 
\begin{align}
\notag
g_3(x) &= \left(px + \frac{1-p}{3}\right)\sum_{k=1}^3\binom{3-1}{k-1}\mathcal P_{3,k}(x) + \frac{2(1-p)}{3}\sum_{k=2}^3\binom{3-2}{k-2}\mathcal Q_{3,k}(x) \\
\notag
&=\left(px + \frac{1-p}{3}\right)\left[ \mathcal P_{3,1}(x) + 2 \mathcal P_{3,2}(x) + \mathcal P_{3,3}(x)\right]\\ 
\label{eqn:defg3}
& \qquad +\frac{2(1-p)}{3}\left[\mathcal Q_{3,2}(x) + \mathcal Q_{3,3}(x)\right],
\end{align}
in which each term is related to a $P$- or $Q$- polynomial via equations \eqref{eqn:relation1} and \eqref{eqn:relation2}. \\
Secondly, the computation of the $\SFM(2,p)$ case gives the $P$- and $Q$- polynomials
\begin{align*}
&P_{1}(z_1) = z_1; \quad P_{2}(z_1, z_2) = z_2(z_2-z_1);\quad Q_{2}(z_1,z_2) = z_2^2,
\end{align*}
whereas the polynomials $P_{3}$ and $Q_{3}$ can be easily deduced from \eqref{eqn:p-recursion} and \eqref{eqn:q-recursion}, i.e., 
\begin{align*}
P_{3}(z_1,z_2, z_3) &= z_3^3 - z_3^2 P_{1}(z_1) - 2z_3 P_{2}(z_1,z_2)\\
&=z_3^3 - z_3^2 z_1 - 2z_3 z_2(z_2-z_1);\\
Q_{3}(z_1,z_2, z_3) &= z_3^3 - z_3 Q_{2}(z_1, z_2)\\
&=z_3^3 - z_3z_2^2. 
\end{align*}
Finally, combining \eqref{eqn:relation1}, \eqref{eqn:relation2} and \eqref{eqn:defg3}, we have
\begin{align*}
    g_3(x) &=\left(px + \frac{1-p}{3}\right)\left\{z_3^3 +(1-z_3^2) z_1 + 2z_2(z_2-z_1)(1-z_3)\right\}\\ 
& \qquad +\frac{2(1-p)}{3}\left[z_3^3 + (1-z_3) z_2^2  \right], 
\end{align*}
where $z_k = g_3\circ c_3^{(k-1)}(x) = g_3\left(\frac{3px + (k-1)(1-p)}{p+2}\right)$ for $k=1,2,3$. This is the desired self-consistency equation for $g_3$.
\end{example}
\begin{remark}
The corresponding operator $\mathcal A_{3}=\mathcal A_{3,p}$ in this case can be defined according to \eqref{eq:Adp-def}: for any $h \in \mathcal I$, $\mathcal A_{3}$ maps $h$ to a function $\mathcal A_{3}h$ whose value at any $x\in [0,1)$ is given by
\begin{align*}
(\mathcal A_{3}h)(x) &=\left(px + \frac{1-p}{3}\right)\sum_{k=1}^3\binom{3-1}{k-1}P_{k}(Z_{1}(h)(x), \ldots, Z_{k}(h)(x))\\ 
&\quad + \frac{2(1-p)}{3}\sum_{k=2}^3\binom{3-2}{k-2}Q_{k}(Z_{1}(h)(x), \ldots, Z_{k}(h)(x)),
\end{align*}
where, as defined in \eqref{eqn:def-zk} and \eqref{eqn:constC}, 
$$Z_{k}(h)(x) =h\circ c_{3,p}^{(k-1)}(x)= h\left(\frac{px+(1-p)(k-1)/3}{p+2(1-p)/3}\right), \quad \text{for }k=1,2,3. $$
 
\end{remark}
\subsection{Proof of Proposition \ref{prop:recursive}} It suffices to prove Lemma \ref{lem:first-d-terms} and the recursive relations that the $P$- and $Q$-polynomials satisfy. 
\begin{proof}[Proof of Lemma \ref{lem:first-d-terms}] Starting from the definition of $\mathcal P_{d, k}(x)$ in \eqref{eqn:def-Pdk}, we will first show that 
\[
\mathcal P_{d,k}(x):= \E\left(x^{\sum_{j=1}^k V_{d, j\to \rootvertex}^{\ast}}\mathbf 1_{A_{d, k}} | f_\rootvertex \to o_1, D_1\right)
\]
can be written as a polynomial of 
$(g_{d}\circ c^{(0)}_{d}(x), \ldots, g_{d}\circ c^{(k-1)}_{d}(x))$.  For any $o_i \in  \{o_1,\ldots, o_d\}$, define $J(o_{i})$ to be the subset 
$$
J(o_i) :=  \{o_j \ne o_i: \text{some frog initially placed in }\mathbb T_d(o_i) \text{ has visited } o_j\} \subset \{o_{1}, \ldots, o_{d}\}.
$$
By slightly abusing the notation, we write $J(A)=\cup_{o_i \in A} J(o_i)$ for any $A\subset \{o_1,\ldots, o_d\}$.\\
At the beginning, only subtree $\mathbb T_d(o_1)$ is activated. After that, some frogs in $\mathbb T_d(o_1)$ may enter some other subtrees $\mathbb T_d(o_j)$ with $j\ne 1$ and wake up frogs there, and the newly-wakened frogs may continue to explore and activate new branches. Eventually, if event $A_{d,k}$ occurs, we have, exactly the subtrees $\mathbb T_d(o_1),\ldots, \mathbb T_d(o_k)$ are visited but none of the subtrees $\mathbb T_d(o_{k+1}),\ldots, \mathbb T_d(o_d)$ are ever visited.
Thus on the event $\{f_\rootvertex \to o_1\}\cap D_1$, $A_{d, k}$ is equivalent to
\begin{align}
\label{eqn:Adk-condition}
\bigcup_{m=0}^{k-1} J^{(m)}(o_1) = \{o_1,\ldots, o_k\}.
\end{align}
where $J^{(m)}$ denotes the $m$-fold composition of $J$ (the convention is $J^{(0)}(A) = A$). Note that we only need to do at most $(k-1)$ compositions here because if $J^{(m)}(o_1)$ does not contain any new vertex that did not appear previously in $J^{(0)}(o_1), \ldots, J^{(m-1)}(o_1)$, then no more new subtrees can be activated thereafter (i.e., by applying the  function $J$ even more times). We can thus write $\mathcal P_{d, k}$ as
\begin{align}
\label{eqn:Pdk-visited}
    \mathcal P_{d, k}(x) = &\sum_{\substack{J_{1}\ldots J_{k}\\(\star)}}\ \ \prod_{m=1}^{k}\E\left(x^{V^\ast_{d,m\to \rootvertex}}\mathbf 1_{J(o_m)=J_m}|\mathbb T_d(o_m)\text{ is ever visited}\right),
\end{align}
where the summation $(\star)$ is over all choices of $J_1, \ldots, J_k \subset \{o_1, \ldots, o_k\}$ such that if $J(o_m)=J_m$ for $m=1,2,\ldots, k$, then \eqref{eqn:Adk-condition} is satisfied. Each choice of $J_{1}, \ldots, J_{k}$ defines a pattern of frog flows among the $d$ branches attached to $\rootvertex'$. The above expression is because the exact number of visits to the root vertex $\rootvertex$ by frogs in $\mathbb T_d(o_m)$ only depends on two events: (a) whether the subtree $\mathbb T_d(o_m)$ is activated and (b) whether some frogs in $\mathbb T_d(o_m)$ have entered some other subtrees (thus these frogs can not visit the root $\rootvertex$ due to non-backtraking). For each term in \eqref{eqn:Pdk-visited}, we re-write the event $\{J(o_m)=J_m\}$ in terms of the events $B_d(o_m; J)$, whose definition is in \eqref{eqn:eventB}. We can immediately see the following equivalence: for any set $J\subset \{o_1,\ldots, o_d\}\setminus\{o_m\}$,  \begin{align*}
    B_d(o_m; J) &:=\left\{\begin{array}{l}
\text{frogs originally placed on vertices of }\\ 
 \mathbb T_{d}(o_{m}) \text{ never enter subtrees } \mathbb T_{d}(o_j), o_j \in J
\end{array}\right\}\\
    &= \{J(o_m)\subseteq \{o_1,\ldots, o_{m-1}, o_{m+1},\ldots, o_d\}\setminus J\}.
\end{align*} Now for any $J_m \subset \{o_1,\ldots, o_d\}$, let $J_{m,d}^c := \{o_1,\ldots,  o_{m-1}, o_{m+1},\ldots, o_d\}\setminus J_m$, and we write the subscript $d$ to emphasize that the compliment is taken in the case $\SFM(d, p)$. Then,
\begin{align*}
    \{J(o_m)=J_m\} & =\left\{ 
    \begin{array}{c}
    \text{frogs in $\mathbb T_d(o_m)$ have never}\\
    \text{visited any of }(\mathbb T_d (o_{j}))_{o_j \in J_{m,d}^c}
    \end{array}\right\}  \bigcap  \left\{ 
    \begin{array}{c}
    \text{frogs in $\mathbb T_d(o_m)$ {\em have} visited}\\
    \text{every tree in }(\mathbb T_d (o_{j}))_{o_j\in J_m}
    \end{array}\right\} \\
    &= B_d(o_m; J_{m,d}^c)   \bigcap  \left\{ 
    \begin{array}{c}
    \text{frogs in $\mathbb T_d(o_m)$ have not visited}\\
    \text{at least one tree in }(\mathbb T_d (o_{j}))_{o_j\in J_m}
    \end{array}\right\}^{c} \\
    &= B_d(o_m; J_{m,d}^c) \setminus  \left(\bigcup_{a \in J_{m}} B_d(o_m;\{a\}) \right)\\
    &= B_d(o_m; J_{m,d}^c) \setminus   \left(\bigcup_{a \in J_{m}} B_d(o_m;J_{m,d}^{c}\cup \{a\}) \right),
\end{align*}
where we used the fact that
\[
B_d(o_m, A) \cap B_d(o_m, B) = B_d(o_m, A\cup B). 
\]
The union on the right can be re-written using the intersection of these $B$-sets according to the inclusion-exclusion principle, and we thus have
\begin{align*}
\mathbf 1_{\{J(o_m)=J_m\}} &=\mathbf 1_{B_d(o_m; J_{m,d}^c)}  - \sum_{a_1\in J_m} \mathbf 1_{B_d(o_m; J_{m,d}^c\cup \{a_1\})} + \sum_{\substack{a_1,a_2\in J_m\\ a_1\ne a_2}} \mathbf 1_{B_d(o_m; J_{m,d}^c\cup \{a_1, a_2\})}  \\
&\qquad - \sum_{\substack{a_1, a_2, a_3\in J_m\\ \text{distinct}}}\mathbf 1_{B_d(o_m; J_{m,d}^c\cup\{a_1, a_2, a_3\})}+ \cdots + (-1)^{|J_m|}\mathbf 1_{B_d(o_m; J_{m,d}^c\cup J_m)}\\
&= \mathbf 1_{B_d(o_m; J_{m,d}^c)} +  \sum_{\substack{S\subset J_m\\ \text{nonempty}}} (-1)^{|S|}\mathbf 1_{B_d(o_m; J_{m,d}^c\cup S)} = \sum_{S\subset J_m} (-1)^{|S|}\mathbf 1_{B_d(o_m; J_{m,d}^c\cup S)}.
\end{align*}
Therefore, \eqref{eqn:Pdk-visited} can be expressed as
\begin{align*}
     \mathcal P_{d, k}(x) &= \sum_{\substack{J_1\ldots J_k\\ (\star)}}\ \prod_{m=1}^{k}\sum_{S_m \subset J_m}(-1)^{|S_m|}\E\left(x^{V^\ast_{d, m\to \rootvertex}}\mathbf 1_{B_d(o_m; J_{m,d}^c \cup S_m)}|\mathbb T_d(o_m)\text{ is ever visited}\right)\\
     &=\sum_{\substack{J_1\ldots J_k\\(\star)}}\ \prod_{m=1}^{k}\sum_{S_m \subset J_m}(-1)^{|S_m|} g_d\circ c_d^{(d-1-|J_{m,d}^c \cup S_m|)}(x),\quad \text{(by Lemma }\ref{lem:binomial}).
     \end{align*}
 Since  $|J_{m, d}^c|+|J_m|= d-1$ by definition and $J_{m,d}^c$ and $S_m$ are disjoint, we have
 \[
d-1-|J_{m,d}^c \cup S_m| = d-1-|J_{m,d}^{c}|-|S_{m}|=|J_{m}|-|S_{m}|
\]
and it then follows
     \begin{align*}
      \mathcal P_{d, k}(x) &=\sum_{\substack{J_1,\ldots, J_k\\ (\star)}}\prod_{m=1}^{k}\sum_{S_m \subset J_m}(-1)^{|S_m|} g_d\circ c_d^{(|J_{m}|-|S_m|)}(x),
\end{align*}
Since $J_m\subset \{o_1,\ldots, o_k\}\setminus\{o_m\}$ then we have $0\le |J_m|, |J_{m}|-|S_{m}|\le k-1$. The above is a polynomial of $g_{d}\circ c_{d}^{(i)}(x)$ for $i=0, 1,\ldots, k-1$.
Now, following exactly the same argument, we can also get
\begin{align*}
     \mathcal P_{d+1, k}(x) &= \E\left(x^{\sum_{j=1}^k V_{d+1, j\to \rootvertex}^{\ast}}\mathbf 1_{A_{d+1, k}} | f_\rootvertex \to o_1, D_1\right)\\
     &=\sum_{\substack{J_1,\ldots, J_k\\ (\star\star)}}\prod_{m=1}^{k}\sum_{S_m \subset J_m}(-1)^{|S_m|} g_{d+1}\circ c_{d+1}^{(d+1-1-|J_{m,d+1}^c \cup S_m|)}(x)\\
     &=\sum_{\substack{J_1,\ldots, J_k\\ (\star\star)}}\prod_{m=1}^{k}\sum_{S_m \subset J_m}(-1)^{|S_m|} g_{d+1}\circ c_{d+1}^{(|J_{m}| - |S_m|)}(x)
\end{align*}
where the condition for $(J_1, \ldots, J_k)$ to satisfy in the summation $(\star \star)$ is the same as that in \eqref{eqn:Pdk-visited}: this is because on $A_{d+1, k}$ when only subtrees $\mathbb T_d(o_1),\ldots, \mathbb T_d(o_k)$ are activated, the possible choices for $J_1, \ldots, J_k$ are still restricted to those subsets of $\{o_1,\ldots, o_k\}$ such that \eqref{eqn:Adk-condition} is satisfied. The only difference is that, for the $\SFM(d+1, p)$ case, the complement for each $J_m$ needs to be taken with respect to the set $\{o_1,\ldots, o_{m-1}, o_{m+1},\ldots, o_{d+1}\}$, that is
\[
J_{m, d+1}^c = \{o_1,\ldots, o_{m-1}, o_{m+1},\ldots, o_{d+1}\} \setminus J_m,
\]
and so now we have $|J_{m, d+1}^c| + |J_m| = d$.  This suggests that $\mathcal P_{d+1,k}(x)$ can be written as a polynomial of $g_{d+1}\circ c_{d+1}^{(i)}(x)$ for $0\le i \le k-1$.

We can now conclude that, by comparing the expression of $\mathcal P_{d,k}(x)$ and that of  $\mathcal P_{d+1, k}(x)$, if 
$
P_k(z_1,\ldots, z_k) 
$ is a polynomial in $(z_1,\ldots, z_k)$ such that
\[
\mathcal P_{d,k}(x) = P_k(g_d\circ c_d^{(0)}(x),\ldots, g_d\circ c_d^{(k-1)}(x)), 
\]
then replacing the variables from $\big(g_{d}\circ c_{d}^{(i)}(x)\big)_{0\le i\le k-1}$ to $\big (g_{d+1}\circ c_{d+1}^{(i)}(x)\big)_{0\le i\le k-1}$, we can obtain
\[
\mathcal P_{d+1,k}(x) = P_k(g_{d+1}\circ c_{d+1}^{(0)}(x),\ldots, g_{d+1}\circ c_{d+1}^{(k-1)}(x)). 
\]
The proof for the $Q$-polynomials follows essentially the same arguments, except that for the $Q$-polybomials, condition \eqref{eqn:Adk-condition} would become
\[
\bigcup_{m=0}^{k-1} J^{(m)}(\{o_1, o_2\}) = \{o_1,\ldots, o_k\}.
\]
We thus omit the details. 
\end{proof}
\begin{proof}[Proof of Proposition \ref{prop:recursive}] Having shown that $\mathcal P_{d,k}(x)$ and $\mathcal Q_{d,k}(x)$ are polynomials of $g_{d}\circ c_{d}^{(i)}(x)$, $i=0, \ldots, (k-1)$, it suffices to prove that these polynomials $P_{k}: \mathbb R^{k} \to \mathbb R$ and $Q_{k}:\mathbb R^{k}\to \mathbb R$, which determine $\mathcal P_{d,k}(x)$, $\mathcal Q_{d,k}(x)$ via \eqref{eqn:relation1}, \eqref{eqn:relation2},  satisfy the recursive relations \eqref{eqn:p-recursion} and \eqref{eqn:q-recursion}.
Again, since the arguments for proving these two equations are essentially the same, we only provide the details for the proof of \eqref{eqn:p-recursion}. In view of Lemma \ref{lem:first-d-terms}, it suffices to show that for $d = k+1$, we have
\begin{align}
\label{eqn:lem10-to-prove}
\mathcal P_{k+1, k+1}(x) = [g_{k+1}\circ c_{k+1}^{(k)}(x)]^{k+1} - \sum_{l=1}^k\binom{k}{l-1} [g_{k+1}\circ c_{k+1}^{(k)}(x)]^{k+1-l}\mathcal P_{k+1, l}(x). 
\end{align}
Thus if $P_{k+1}$ is defined according to \eqref{eqn:p-recursion}, then for all $d\ge k+1$, 
\[
\mathcal P_{d, k+1}(x)=P_{k+1}\big(g_{d}\circ c_{d}^{(0)}(x), \ldots, g_{d}\circ c_{d}^{(k)}(x)\big).
\]
To do this, we would like to add to the self-similar frog model $\SFM(k+1, p)$ a second stage of frog re-activation. 
\begin{itemize}
    \item Stage I: Run an ordinary $\SFM(k+1, p)$, starting from one activated frog placed at the root vertex. For $i=1,2,\ldots$, let $\tau_i$ be the first time that $o_i$ is visited by some active frog. Set $\tau_i=\infty$ if $o_i$ is never visited {\em in this stage}. \vspace{1mm}
    \item Stage II: For {\bf each} $i=1,2,\ldots, k+1$, if $\tau_i=\infty$, introduce another activated frog at $\rootvertex'$, have it move to $o_i$ and wake up the sleeping frog there. These activated  frogs then perform independent non-backtracking random walks and activate sleeping frogs according to the rules of $\SFM(k+1, p)$.
\end{itemize}
We call this process a re-activated self-similar frog model (rSFM) and denote by $\tilde V_{k+1, i\to \rootvertex}$ the number of visits to the root vertex $\rootvertex$ made by frogs initiated in the subtree $\mathbb T_{k+1}(o_i)$ via $\rootvertex'$. Let $\mathcal L\subseteq\{o_1,\ldots, o_{k+1}\}$ be the branches activated in Stage I and set $L:=|\mathcal L|$, that is
\[
L = \sum_{i=1}^d \mathbf 1\{\tau_i < \infty\}.
\]
Since the first stage is an ordinary $\SFM(k+1,p)$, we have
\begin{align*}
& \ \ \mathcal P_{k+1, k+1}(x) = \E \left(x^{\sum_{i=1}^{k+1}V^{\ast}_{k+1, i\to \rootvertex}} \mathbf 1_{A_{k+1, k+1}}| f_{\rootvertex} \to o_{1}, D_{1}\right)\\
&=\E \left(x^{\sum_{i=1}^{k+1}\tilde V_{k+1, i\to \rootvertex}} \mathbf 1_{\{L=k+1\}}| f_{\rootvertex} \to o_{1}, D_{1}\right)\\
&=\E \left(x^{\sum_{i=1}^{k+1}\tilde V_{k+1, i\to \rootvertex}}| f_{\rootvertex} \to o_{1}, D_{1}\right) - \sum_{l=1}^{k}\E \left(x^{\sum_{i=1}^{k+1}\tilde V_{k+1, i\to \rootvertex}} \mathbf 1_{\{L=l\}}| f_{\rootvertex} \to o_{1}, D_{1}\right)\\
&=\E \left(x^{\sum_{i=1}^{k+1}\tilde V_{k+1, i\to \rootvertex}}| f_{\rootvertex} \to o_{1}, D_{1}\right) - \sum_{l=1}^{k}\binom{k}{l-1}\E \left(x^{\sum_{i=1}^{k+1}\tilde V_{k+1, i\to \rootvertex}} \mathbf 1_{\{\mathcal L = \{o_{1}, \ldots, o_{l}\}\}}| f_{\rootvertex} \to o_{1}, D_{1}\right).
\end{align*}
In rSFM,  all branches attached to $\rootvertex'$ are activated, either in Stage I or in Stage II. Moreover, since only one frog is allowed to enter each subtree $\mathbb T_{d}(o_{i})$, we still have the self-similar structure that  the number of visits from $o_i$ to $\rootvertex'$ is an independent copy of $V^\ast_d$; it has nothing to do with the way the subtree $\mathbb T_{k+1}(o_i)$ is activated (i.e., by which frog or in which stage). It follows that for $n\ge 1$
$$\left(\tilde V_{k+1, i\to \rootvertex}|\tilde V_{k+1, i\to \rootvertex'}=n\right) \sim \text{Bin}\left(n,\frac{p}{p+\frac{k}{k+1}(1-p)} \right).$$
Following the same computation as in the SFM case (Lemma \ref{lem:binomial}), we get in rSFM 
$$\E \left(x^{\tilde V_{k+1, i\to \rootvertex}}\right) = g_{k+1}\circ c_{k+1}^{(k)}(x), \quad \text{for all } i=1,2,\ldots k+1.$$
Therefore,
\begin{align*}
& \mathcal P_{k+1, k+1}(x) \\
&=\E \left(x^{\sum_{i=1}^{k+1}\tilde V_{k+1, i\to \rootvertex}}\right) \\
&\qquad - \sum_{l=1}^{k}\binom{k}{l-1}\E \left(x^{\sum_{i=1}^{l}\tilde V_{k+1, i\to \rootvertex}} \mathbf 1_{\{\mathcal L = \{o_{1}, \ldots, o_{l}\}\}}| f_{\rootvertex} \to o_{1}, D_{1}\right)\E(x^{\sum_{i=l+1}^{k+1} \tilde V_{k+1,i\to \rootvertex}})\\
&=[g_{k+1}\circ c_{k+1}^{(k)}(x)]^{k+1}\\
&\qquad - \sum_{l=1}^{k}\binom{k}{l-1}\E \left(x^{\sum_{i=1}^{l}V^{\ast}_{k+1, i\to \rootvertex}} \mathbf 1_{A_{d,l}}| f_{\rootvertex} \to o_{1}, D_{1}\right)[g_{k+1}\circ c_{k+1}^{(k)}(x)]^{k-l+1},
\end{align*}
which is exactly \eqref{eqn:lem10-to-prove}.
\end{proof}

\section{The recurrence of $\FM(d, 1/3)$ \label{sec:propertyAd}}
In this section, we prove Propostion \ref{prop:key}, i.e., $\SFM(d, p)$ for $p^\ast(d, 1/3):=\frac{d-1}{2d-1}$ is recurrent. By Proposition \ref{prop:keycompare}, this implies $\FM(d, 1/3)$ is recurrent, which finishes the proof of Theorem \ref{thm:upbd}. The special case that $\SFM(2, 1/3)$ is recurrent was proved in \cite{HJJ1}, and our goal is to show that for any $d\ge 3$, $\SFM(d, \frac{d-1}{2d-1})$ is at least ``as  recurrent'' as $\SFM(2, 1/3)$. Throughout this section, the drift parameter is always set to $p=p^\ast(d, 1/3) = \frac{d-1}{2d-1}$. Since the model $\SFM(d, \frac{d-1}{2d-1})$ is parametrized by $d$ only, we abbreviate the notation $X_{d,
\frac{d-1}{2d-1}}$ as $X_{d}$, for $X=g, \mathcal A, V^{\ast}$ or $c^{(k)}$. 

With $p=\frac{d-1}{2d-1}$, the linear functions $c_{d}^{(k)}(x)$ become
\begin{align}
\label{eqn:specialC}
    c_d^{(k)}(x) = c_{d, \frac{d-1}{2d-1}}^{(k)}(x)= \frac{\frac{(d-1)x}{2d-1} + \frac{k}{2d-1}}{\frac{d-1}{2d-1}+ \frac{d-1}{2d-1}} = \frac{x}{2}+ \frac{k}{2(d-1)}\quad \text{for } k = 0, 1, \ldots, (d-1).
\end{align}
The operator $\mathcal A_d$ on $\mathcal I=\{f: [0, 1) \to [0, 1], \text{ nondecreasing}\}$ can be defined according to \eqref{eq:Adp-def}:  $\forall h \in \mathcal I$, $\mathcal A_d = \mathcal A_{d, \frac{d-1}{2d-1}}$ maps $h$ to 
\begin{align}
  \notag
\mathcal A_d h (x) &:=  \frac{(d-1)x + 1}{2d-1} \sum_{k=1}^d \binom{d-1}{k-1} P_k\left(h\left(\frac{x}{2}\right),\ldots, h\left(\frac{x}{2}+\frac{k-1}{2(d-1)}\right)\right) \\ 
 \label{eqn:def-Ad}
 &\quad +  \frac{d-1}{2d-1} \sum_{k=2}^d \binom{d-2}{k-2} Q_k\left(h\left(\frac{x}{2}\right),\ldots, h\left(\frac{x}{2}+\frac{k-1}{2(d-1)}\right)\right)
\end{align}
and $g_d \in \mathcal I$ is a fixed point of $\mathcal A_d$ by Proposition \ref{prop:recursive}.

In the base case $d=2$ and $\frac{d-1}{2d-1}=\frac{1}{3}$, the operator $\mathcal A_2$ maps any function $h\in \mathcal I$ to $\mathcal A_{2}h$, defined as 
\begin{align*}
\mathcal A_2 h (x) &:= \frac{x+2}{3} h\left(\frac{x+1}{2}\right)^2+ \frac{x+1}{3}h\left(\frac{x}{2}\right)\left[1-h\left(\frac{x+1}{2}\right)\right],
\end{align*}
and $g_2(x) := \E (x^{V^\ast_{2}}) \in \mathcal I$  is a 
fixed point of $\mathcal A_2$. It was shown in \cite{HJJ1} that this operator $\mathcal A_2$ exhibits a few nice properties:
\begin{itemize}
    \item (Closed) For any $h\in \mathcal I$, $\mathcal A_2h \in \mathcal I$; 
    \item (Monotone) For any $h_1, h_2 \in \mathcal I$ with $h_1\le h_2$,  $\mathcal A_2 h_1 \le \mathcal A_2 h_2$;
    \item (Vanishing) $\lim_{n\to \infty} \mathcal A_2^n 1 = 0$.
\end{itemize}
It follows that 
$
g_{2} = \mathcal A_2^n g_2 \le \mathcal A_2^n 1 \to 0
$, meaning  $g_{2}\equiv 0$ and $V^\ast_{2} = \infty$ almost surely. This proves that SFM$(2, 1/3)$ is recurrent.

\begin{prop}\label{prop:Ad-A2}
For any $d \ge 2$, let $g_d(x)=\E(x^{V_{d}^{\ast}})$ be the probability generating function of $V_{d}^\ast$, the number of visits to the root in $\SFM(d, \frac{d-1}{2d-1})$. Then $\mathcal A_{d} g_d \le \mathcal A_2 g_d$, where $\mathcal A_{d}$ is an operator, whose domain is $\mathcal I=\{f: [0,1)\to[0,1], \text{nondecreasing}\}$, defined by \eqref{eqn:def-Ad}.
\end{prop}
Proposition \ref{prop:Ad-A2} indicates that all self-similar frog models $\SFM(d, \frac{d-1}{2d-1})$ are at least as recurrent as $\SFM(2, 1/3)$.  To see this, we make use of the fact that $g_{d}$ is a fixed point of $\mathcal A_{d}$ and the three properties of $\mathcal A_{2}$ repeatedly:
\[
g_d = \mathcal A_d g_d \le \mathcal A_2  g_d = \mathcal A_2(\mathcal A_d g_d) \le \mathcal A_2 \mathcal A_2 g_d \le \cdots\le  \mathcal A_2^{n} g_d\le \mathcal A_2^n 1 \to 0,
\]
which implies $g_{d}\equiv 0$ and $V^\ast_{d}=\infty$ almost surely. 

\begin{proof}[Proof of Proposition \ref{prop:Ad-A2}] Fix $d \ge 2$ and consider the last conditional expectations $\mathcal P_{d,d}$ and $\mathcal Q_{d,d}$ in the self-consistency equation of $g_d$. To simplify the notation, we fixed an arbitrary $x\in [0,1)$ and write 
\[
z_k = g_d\circ c_d^{(k-1)}(x) = g_d\left(\frac{x}{2} + \frac{k-1}{2(d-1)}\right), \quad k=1,2\ldots, d.
\]
Then $(z_k)_{1\le k \le d}$ is an increasing sequence in $[0, 1]$.  By Proposition \ref{prop:recursive} and Lemma \ref{lem:first-d-terms},
\begin{align*}
    0 &\le \E\left(x^{\sum_{j=1}^d V_{d,j\to \rootvertex}^\ast}\mathbf 1_{A_{d, d}} | f_\rootvertex \to o_1, D_1\right) = \mathcal P_{d, d}(x) = P_d(z_1,\ldots, z_d)\\
    &= z_{d}^{d} - \sum_{l=1}^{d-1} \binom{d-1}{l-1}z_{d}^{d-l}P_{l}(z_1,\ldots, z_l)\\
    &= z_d \left[z_{d}^{d-1} - \sum_{l=1}^{d-1} \binom{d-1}{l-1}z_{d}^{d-1-l}P_{l}(z_1,\ldots, z_l)\right]\\
    &\le z_{d}^{d-1} - \sum_{l=1}^{d-1} \binom{d-1}{l-1}z_{d}^{d-1-l}P_{l}(z_1,\ldots, z_l) \quad (\text{replace the  factor $z_d$ outside by 1})\\
    &= z_{d}^{d-1} - \sum_{l=1}^{d-2} \binom{d-1}{l-1}z_{d}^{d-1-l}P_{l}(z_1,\ldots, z_l) - \binom{d-1}{d-2}P_{d-1}(z_1, \ldots, z_{d-1})\\
    &\le z_{d}^{d-2} - \sum_{l=1}^{d-3} \binom{d-1}{l-1}z_{d}^{d-2-l}P_{l}(z_1,\ldots, z_l) - \binom{d-1}{d-3}P_{d-2}(z_1,\ldots, z_{d-2}) -  \binom{d-1}{d-2}P_{d-1}(z_1, \ldots, z_{d-1})\\
    &\le \cdots \le z_d^2 - z_d P_1(z_1) -\sum_{l=2}^{d-1}\binom{d-1}{l-1} P_l(z_1, \ldots, z_l)\\
    &=z_d^2 - z_d z_1 - \sum_{l=2}^{d-1}\binom{d-1}{l-1} P_l(z_1, \ldots, z_l),
\end{align*}
where we keep plugging out the common factor $z_{d}$ and replacing it by 1. This means 
\[
\sum_{l=1}^d\binom{d-1}{l-1} P_l(z_1,\ldots, z_l) = z_1 + \sum_{l=2}^{d-1}\binom{d-1}{l-1} P_l(z_1, \ldots, z_l)  + P_d(z_1, \ldots, z_d)  \le z_1+ z_d^2 - z_dz_1. 
\]
We can get similar result for the $Q$-polynomials following the same strategy:
\[
\sum_{k=2}^d\binom{d-2}{k-2}Q_l(z_1,\ldots, z_l)\le z_d^2.
\]
Combining the above two inequalities, we get
\begin{align*}
    \mathcal A_d g_d(x) &=\frac{(d-1)x+1}{2d-1}\sum_{k=1}^d \binom{d-1}{k-1}P_k(z_1,\ldots, z_k) + \frac{d-1}{2d-1} \sum_{k=2}^d \binom{d-2}{k-2} Q_k\left(z_1,\ldots, z_k\right)\\
    &\le \frac{(d-1)x+1}{2d-1} (z_1 + z_d^2 - z_d z_1 )  + \frac{d-1}{2d-1} z_d^2.
\end{align*}
Noticing that for $x\in [0,1)$ and $d\ge 2$
\begin{align*}
\frac{(d-1)x+1}{2d-1} - \frac{x+1}{3} &= \frac{(3d-3)x+3 - (2d-1)x - (2d-1)}{3(2d-1)} \\
&=\frac{(d-2)(x-2)}{3(2d-1)} \le \frac{-(d-2)}{3(2d-1)} = \frac{1}{3} - \frac{d-1}{2d-1},
\end{align*}
then
\begin{align*}
    \mathcal A_d g_d(x) &\le \frac{x+1}{3} (z_1 + z_d^2 - z_dz_1) + \frac{1}{3} z_d^2 -\frac{d-2}{3(2d-1)}(z_1+z_d^2 - z_dz_1 -z_d^2)\\
    &\le \frac{x+1}{3} (z_1 + z_d^2 - z_dz_1) + \frac{1}{3} z_d^2.
\end{align*}
Observe that for all $d\ge 2$, the exact forms of the first and last linear functions $c_{d}^{(0)}$ and $c_{d}^{(d-1)}$ are the same for all $d\ge 2$; see equation \eqref{eqn:specialC}:
\[
c_d^{(0)}(x) = \frac{x}{2} \quad \text{and} \quad  c_d^{(d-1)}(x) = \frac{x+1}{2}.
\]
This gives
\[
z_1 = g_d\left(\frac{x}{2}\right) \quad \text{and }\quad z_d = g_d\left(\frac{x+1}{2}\right),
\]
and thus 
\[
 \mathcal A_d g_d(x) \le \frac{x+2}{3} g_{d}\left(\frac{x+1}{2}\right)^2+ \frac{x+1}{3}g_{d}\left(\frac{x}{2}\right)\left[1-g_{d}\left(\frac{x+1}{2}\right)\right] = \mathcal A_{2}g_{d}(x).
\]
The above holds for any $x\in [0, 1)$, and thus the proof of Proposition \ref{prop:Ad-A2} is complete.
\end{proof}

\bibliographystyle{amsalpha}
\bibliography{frog_paper_cover.bib}

\end{document}